\documentclass[11pt, a4paper, oneside]{article}
\usepackage{amsmath, amssymb, amsthm, amsfonts, amsxtra, latexsym, amscd,
pb-diagram,  graphics, hyperref, setspace}
\usepackage [all]{xy}
\theoremstyle{plain}

\newcommand{\ri}{\rightarrow }

\newcommand{\Fm}{\widetilde{F}}

\DeclareMathOperator{\Coker}{Coker}\DeclareMathOperator{\Dis}{Dis}
 \DeclareMathOperator{\Hom}{Hom}
 
\DeclareMathOperator{\Ext}{Ext} \DeclareMathOperator{\Ker}{Ker}

\newtheorem{thm}{\bf Theorem}
\newtheorem{lem}[thm]{\bf Lemma} 
\newtheorem{pro}[thm]{\bf Proposition} 
\newtheorem{hq}[thm]{\bf Corollary} 

\theoremstyle{definition}

\begin{document}

\centerline{\Large  \bf Abelian crossed modules and strict}
\centerline{\Large  \bf  Picard categories} 
\vspace{0.5cm}
 \centerline{\bf
Nguyen Tien Quang}
 \centerline{\it Department of Mathematics, Hanoi National University of
Education,}
 \centerline{\it E-mail:
cn.nguyenquang@gmail.com} \vspace{0.25cm} \centerline{\bf Che Thi
Kim Phung} \centerline{\it Faculty of Mathematics and Applications,
Saigon University,} \centerline{\it E-mail: kimphungk25@yahoo.com}
\vspace{0.25cm}
 \centerline{\bf Ngo Sy Tung} \centerline{\it Department of Mathematics, Vinh University,}
\centerline{\it E-mail: ngositung@yahoo.com}

\begin{abstract}
In this paper, we state the notion of morphisms in the category of abelian crossed modules and prove that  this category is equivalent to the category of strict Picard categories and regular symmetric monoidal functors. The theory of obstructions for symmetric monoidal functors and symmetric cohomology groups are applied to show  a treatment of the group extension problem of the type of an abelian crossed module. 
\end{abstract}

\section{Introduction}

Crossed modules have been used widely, and in various contexts, since their definition by Whitehead \cite{White49}
in his investigation of the algebraic structure of second relative homotopy groups.
 A brief summary of researches
 related to crossed modules was given in  \cite{CCG} in which  Carrasco et al. obtained interesting results on the category of \emph{abelian} crossed modules. The notion of abelian crossed module was characterized by that of the \emph{center} of a crossed module in  the paper of  Norrie  \cite{N90}.

Crossed modules are essentially the same as \emph{strict categorical groups} (see \cite{J-S, Br76, Baez, N}).
A \emph{strict} categorical group is a  categorical group in which the  associativity, unit constraints are strict
$(\textbf{a}=id, \textbf{l}=id=\textbf{\textbf{r}})$ and, for each object $x$, there is  an object $y$ such that $x\otimes y=1=y\otimes x$. This concept is also called a \emph{$\mathcal G$-groupoid} by Brown and Spencer \cite{Br76}, or a \emph{2-group} by Noohi \cite{N}, or a \emph{strict 2-group} by Baez and Lauda \cite{Baez}.

Brown and Spencer \cite{Br76} (Theorem 1) published a  proof that the category of \emph{$\mathcal G$-groupoids} is equivalent to the category \textbf{CrossMd} of crossed modules (the morphisms in the first category are  functors preserving the group structure, those in the second category are homomorphisms of crossed modules).

Another result on crossed modules,   the group extension problem of the type of a crossed module, was presented by Brown and Mucuk in \cite{Br94} (Theorem 5.2). This problem has attracted the attention of many mathematicians.

 In our opinion, the above date can be considered for abelian crossed modules. At the beginning of Section 3, we show that each abelian crossed module is seen as a strict Picard category (as defined in Section 2). Therefore, we can apply Picard category theory to study abelian crossed modules and obtain results similar to the above results on crossed modules.

The content of the paper consists of two main results.
In Section 3, we
prove that (Theorem 4) the category $\bf Picstr$ of strict Picard categories and
regular  symmetric monoidal functors is equivalent to the category
 $\bf AbCross$ of abelian crossed modules.  Every morphism in the   category  $\bf AbCross$
 consists of a homomorphism $(f_1,f_0): \mathcal M\rightarrow \mathcal M'$ of abelian crossed modules  and an element of the group of symmetric 2-cocycles $Z^2_s(\pi_0\mathcal M,\pi_1\mathcal M')$. This theorem is analogous to Theorem 1 \cite{Br76}.


In Section 4, we study the group extension problem of the type of an abelian crossed module. The theory of obstructions for symmetric monoidal functors is applied to show a treatment of this problem. Each abelian crossed module
$B\stackrel{d}{\rightarrow}D$ defines a strict Picard category $\mathbb P$. The third invariant of $\mathbb P$ is an element $\overline{k}\in H^3_s(\Coker d, \Ker d)$. Then a group homomorphism  $\psi: Q\ri \Coker d$ induces $\psi^\ast k\in Z^3_s(Q, \Ker d)$. Theorem \ref{dlc} shows that the vanishing of $\overline{\psi^\ast k}$ in $H^3_s(Q, \Ker d)$ is necessary and sufficient for there to exist a group extension of the type of  an abelian crossed module $B\stackrel{d}{\rightarrow}D$. Each such extension induces a symmetric monoidal functor
$F:Dis Q\ri \mathbb P$. This correspondence determines a bijection  (Theorem \ref{schr})
$$
\Omega:\mathrm{Hom}^{Pic}_{(\psi,0)}[\mathrm{Dis}Q,\mathbb P_{B\ri D}]\rightarrow \mathrm{Ext}^{ab}_{B\ri D}(Q, B,\psi).$$
Theorem \ref{dlc} is analogous to Theorem 5.2 \cite{Br94}.


\section{Preliminaries}

 A {\it symmetric monoidal category}
$\mathbb{P}:=(\mathbb{P},\otimes,I,{\bf a,l,r,c})$
 consists of a category $\mathbb{P}$,
   a functor
$\otimes:\mathbb{P}\times\mathbb{P}\rightarrow\mathbb{P}$  and
natural isomorphisms  ${\bf a}_{X,Y,Z}: (X\otimes Y)\otimes Z
\stackrel{\sim}{\rightarrow}
 X\otimes (Y\otimes Z ),{\bf l}_X: I\otimes X \stackrel{\sim}{\rightarrow} X,
 {\bf r}_X: X\otimes I \stackrel{\sim}{\rightarrow} X$ and ${\bf c}_{X,Y}: X\otimes
  Y\stackrel{\sim}{\rightarrow} Y\otimes X$ such that, for any objects $X,Y,Z,T$ of $\mathbb P$, the following coherence
  conditions hold:

 \vspace{0.1cm}
  \indent $i) \ {\bf a}_{X,Y,Z\otimes T}{\bf a}_{X\otimes Y, Z,T}= (id_X \otimes {\bf a}_{Y,Z,T})
{\bf a}_{X,Y\otimes Z,T}({\bf a}_{X,Y,Z}\otimes id_T)$, \label{2.1}\\
\indent $ii)\
{\bf c}_{X,Y}\cdot
 {\bf c}_{Y,X}=id,$\label{2.1a}\\
\indent $iii)\
(id_X \otimes {\bf l}_Y){\bf a}_{X,I,Y}= {\bf r}_X \otimes id_Y,$ \label{2.2}\\
\indent $iv)\
(id_Y \otimes {\bf c}_{X,Z}) {\bf a}_{Y,X,Z}({\bf c}_{X,Y}\otimes
id_Z)={\bf a}_{Y,Z,X} {\bf c}_{X,Y\otimes Z} {\bf a}_{X,Y,Z}.$
\label{2.3}

\vspace{0.2cm}
A {\it Picard category} is a symmetric monoidal category in which every morphism is invertible and, for each object $X$,  there is an object  $Y$ with  a morphism $X\otimes Y\rightarrow I$.

A  Picard category is said to be \emph{strict} when the constraints {\bf a} = $id$, {\bf c} = $id$, {\bf l} = $id$ = {\bf r} and, for each object $X$,  there is an object  $Y$ such that  $X\otimes Y = I$.

 If $\mathbb P$, $\mathbb P'$ are  symmetric monoidal categories, then a {\it symmetric monoidal functor} $F:=(F,\widetilde{F},F_*): \mathbb P\rightarrow \mathbb P'$ consists of a functor
 $F : \mathbb P \rightarrow \mathbb P'$, natural isomorphisms $\widetilde{F}_{X,Y} :FX \otimes FY \rightarrow F(X\otimes Y)$
  and an  isomorphism
 $F_\ast : I' \rightarrow FI$,
such that, for any objects $X,Y,Z$ of $ \mathbb P$, the following coherence conditions hold:
\begin{equation*}
\widetilde{F}_{X,Y\otimes Z}(id_{FX}\otimes \widetilde{F}_{Y,Z})
{\bf a}_{FX,FY,FZ}= F({\bf a}_{X,Y,Z})\widetilde{F}_{X\otimes Y,Z}(\widetilde{F}_{X,Y}\otimes id_{FZ}), \label{6.5}
\end{equation*}
\begin{equation*}
F({\bf r}_X)\widetilde{F}_{X,I}(id_{FX}\otimes F_\ast)= {\bf
r}_{FX}\ ,\
 F({\bf l}_X)\widetilde{F}_{I,X}(F_\ast\otimes id_{FX}) = {\bf l}_{FX}, \label{6.6}
 \end{equation*}
 \begin{equation*}
\widetilde{F} _{Y,X}{\bf c}_{FX,FY}=F({\bf c}_{X,Y})\widetilde{F}_{X,Y}.
\label{6.7}
\end{equation*}
Note that if $F:=(F,\widetilde{F}, F_* )$ is a symmetric monoidal functor between Picard categories, then the isomorphism $F_*: I'\rightarrow FI$ is implied from $F$ and $\widetilde{F},$ so we can omit $F_*$ when not necessary.

A \emph{symmetric monoidal natural equivalence} between symmetric monoidal functors $(F,\widetilde{F},F_\ast), (F',\widetilde{F}',F'_\ast): \mathbb P\rightarrow \mathbb P'$  is a natural equivalence $\theta:F\stackrel{\sim}{\rightarrow} F'$ such that, for any objects $X,Y$ of $\mathbb P$, the following coherence conditions hold:
\begin{align*}\label{3.4}
\widetilde{F}'_{X,Y}(\theta_X\otimes\theta_Y)=\theta_{X\otimes
Y}\widetilde{F}_{X,Y},\;\theta_I F_\ast=F'_\ast.
\end{align*}

Let $\mathbb{P}:=(\mathbb{P},\otimes,I,{\bf a,l,r,c})$ be a Picard category.  According to Sinh \cite {Sinh75}, $\mathbb{P}$ is equivalent to its reduced Picard category $\mathbb S= S_\mathbb P$
thanks to {\it canonical} equivalences
$$G: \mathbb P\rightarrow \mathbb S,\;\;\;H:\mathbb S\rightarrow \mathbb P.$$
For convenience,
 we briefly recall  the construction of $\mathbb S$. Let $M=\pi_0\mathbb{P}$ be the abelian group of isomorphism classes of
the objects in  $\mathbb{P}$ where the operation is induced by the
tensor product, $N=\pi_1\mathbb{P}$ be the abelian group of
automorphisms of the unit object  $I$ of $\mathbb P$ where the
operation is composition. Then,   objects of $\mathbb S$ are
elements $x\in M$, and its morphisms are automorphisms $(a,
x):x\rightarrow x,\;a\in N$. The composition of morphisms is given
by
$$(a,x)\circ (b,x)=(a+b,x).$$
 The tensor product is defined by
\begin{align*}
 x \otimes y &= x+y,  \\
 (a, x) \otimes (b, y) &= (a+b, x+y).
\end{align*}
The unit constraints in $\mathbb S $ are
strict (in the sense that ${\bf l}_x = {\bf r}_x =id_x$), the
associativity constraint
$\xi$ and the symmetry constraint $\eta$ are, respectively,  functions
$M^3\rightarrow N,\
M^2\rightarrow N $
satisfying normalized condition:
$$\xi(0,y,z)=\xi(x,0,z)=\xi(x,y,0)=0.$$
and satisfying the following relations:

\vspace{0.1cm}
\indent
$i) \ \xi(y,z,t)-\xi(x+y,z,t)+\xi(x,y+z,t)-\xi(x,y,z+t)+\xi(x,y,z)=0,$\\
\indent $ii)\ \eta(x,y)+\eta(y,x)=0,$\\
\indent $iii)\ \xi(x,y,z)-\xi(y,x,z)+\xi(y,z,x)+\eta(x,y+z)-\eta(x,y)-\eta(x,z)=0,$

\vspace{0.15cm}
 The pair $(\xi, \eta)$ satisfying these relations 
  is just an
  element in the group $Z^{3}_{s}(M,N)$ of \emph{symmetric} 3-cocycles in the sense of \cite {MacL52}.
We refer to $\mathbb S$ as {\it Picard  category of type $(M,N)$}.

Let $ \mathbb S=(M,N,\xi, \eta), \mathbb S'=(M',N',\xi',
\eta')$  be Picard categories. A functor $F:  \mathbb
S\rightarrow \mathbb S'$ is called a \emph{functor of  type}
$(\varphi, f)$ if
there are group homomorphisms $\varphi:M\rightarrow M'$,
$f:N\rightarrow N'$ satisfying
$$F(x)=\varphi(x),\ \ F(a,x)=(f(a),\varphi(x)). $$
In this case, $(\varphi ,f)$ is called a \emph{pair of homomorphisms}, and the function
 \begin{equation}\label{ct}
 k=\varphi^{\ast}(\xi',\eta')-f_{\ast}(\xi,\eta)
 \end{equation}
   is called an {\it obstruction} of the functor  $F:\mathbb S\rightarrow
\mathbb S'$   of type $ ( \varphi, f)$.

 The following proposition is implied from the results on monoidal functors of  type
$(\varphi,f)$ in
\cite{QCT} .

\begin{pro}\label{md1} Let $\mathbb P,\mathbb P'$ be Picard categories and $ \mathbb S, \mathbb S'$ be  their reduced Picard categories, respectively.\\
\indent $i)$ Any symmetric monoidal functor $(F, \Fm): \mathbb  P\rightarrow \mathbb P'$
induces a symmetric monoidal functor $\mathbb S_F:\mathbb
S\rightarrow \mathbb S'$ of type $(\varphi, f)$. Further, $\mathbb S_F=G' F H,$ where $H, G'$ are canonical equivalences.\\
\indent $ii)$ Any symmetric monoidal functor $(F,\widetilde{F}):\mathbb S\rightarrow
\mathbb S'$ is a functor of  type $ ( \varphi, f).$\\
\indent $iii)$ The functor  $F:\mathbb S\ri \mathbb S' $ of type $(\varphi, f)$ is realizable, i.e.,
 there are isomorphisms $\widetilde{F}_{x,y} $ so that $(F,\widetilde{F} )$ is  a symmetric monoidal functor, if and only if its obstruction
   $\overline{k}$ vanishes in  $H^3_s(M, N')$. Then, there is a bijection
 \begin{equation*}
   {\mathrm{Hom}}^{Pic}_{(\varphi, f)}[\mathbb S, \mathbb S']\leftrightarrow H^2_s(M,
 N'),
\end{equation*}
 where $\Hom_{(\varphi, f)}^{Pic}[\mathbb S, \mathbb S']$ denotes the set
 of homotopy classes of symmetric monoidal functors of  type  $(\varphi, f)$ from
$\mathbb S$ to $\mathbb S'$.
\end{pro}

\section{Classification of abelian crossed modules by  strict Picard categories }

In this section, we will show a treatment of the problem on
classification of abelian crossed modules due to 2-dimensional symmetric cohomology groups and regular symmetric monoidal functors.

We recall that
a {\it crossed module} is a quadruple $\mathcal M =(B,D,d,\theta),$
where $d:B\ri D,\;\theta:D\ri$ Aut$B$ are group
homomorphisms satisfying the following relations:\\
\indent $C_1.\ \theta d=\mu$,\\
\indent $C_2.\ d(\theta_x(b))=\mu_x(d(b)),\quad x\in D, b\in B$,\\
where $\mu_x$ is an inner automorphism given
by $x$.

In this paper,  the crossed module $(B,D,d,\theta)$ is sometimes
denoted by $B\stackrel{d}{\rightarrow}D$, or simply $B\rightarrow
D$.

Standard consequences of the axioms are that
$\Ker d$ is a left $\Coker d$-module  under the action
$$sa=\theta _x(a),\ \ a\in\Ker d, \ x\in s\in \Coker d.$$
The groups $\Coker d, \Ker d$ are also denoted by $\pi_0\mathcal M,\pi_1\mathcal M$, respectively.

We are interested in the case when $B, D$ are abelian groups. Then, it follows from  the condition $C_1$ that $\theta d = id$ (and hence  Im$d$ acts trivially on $B$). The condition $C_2$ leads to $\theta_x(b)-b \in \Ker d$. Therefore, $\theta $ determines a function $g:\Coker d\times \Ker d\ri \Ker d$  by
$$g(s,b)=sb-b.$$
It is straightforward to see that $g$ is a biadditive normalized function. Conversely,  the data  $(B\xrightarrow{d} D, g)$, where $B, D$ are  abelian, determines completely  a  crossed module. Particularly, if $g=0$ we obtain the notion of \emph{abelian crossed module}. In other words, abelian crossed modules are defined as follows.

\vspace{0.1cm}
\noindent{\bf Definition.} A crossed module $\mathcal M = (B,D,d, \theta)$ is said to be \emph{abelian} when $B, D$ are abelian  and $\theta = 0$.

 For example, if $\mathcal M$ is a crossed module  in which $B,D$ are abelian  and $d$ is a monomorphism, then $\theta=0$. Therefore, $\mathcal M$ is an abelian crossed module.

\vspace{0.2cm}
The notion of abelian crossed modules can be characterized by that of the center of crossed modules as in Norrie's work \cite{N90}. We say that
 the {\it center}
$\xi\mathcal M $ of a crossed module $\mathcal M = (B,D,d,\theta)$ is a subcrossed module of $\mathcal M$ and  defined by  $(B^D,st_D(B)\cap Z(D), d, \theta)$, where $B^D$ is the \emph{fixed point subgroup}  of $B$, $st_D(B)$ is the {\it stabilizer} in $D$ of $B$, that is,
\begin{align*}
B^D&=\{b\in B:\theta_xb=b\;\text{for all} \;x\in D\},\\
st_D(B)&=\{x\in D:\theta_xb=b\;\text{for all} \;b\in B\},
\end{align*}
and  $Z(D)$ is the center of  $D$ (note that $B^D$ is in the center of $B$).
Then, the  crossed module is termed  {\it abelian} if
$\xi(B,D,d,\theta)= (B,D,d,\theta)$.

\vspace{0.1cm}
It is well-known that  crossed modules are the same as strict categorical groups  (see \cite {J-S}, Remark 3.1). Now, we show that abelian crossed modules can be seen as strict Picard categories. We state this in detail.

$\bullet $ For any abelian crossed module $B\rightarrow  D$, we can
construct a strict Picard category $\mathbb P_{B\rightarrow D}=\mathbb P$, called the Picard category {\it associated to} the abelian crossed module  $B\ri D$, as follows.
$$\mathrm{Ob}(\mathbb P)=D,\;\ \mathrm{Hom}(x,y)=\{b\in B\ |\ x=d(b)+y\},$$
for objects $x,y\in D.$
 The composition of two
morphisms is given by
\begin{equation}   (x\stackrel{b}{\ri}y\stackrel{c}{\ri}z)=(x\stackrel{b+c}{\ri}z).\label{htmt} \end{equation}
The tensor operation on objects is given by the addition in
the group $D$ and, for two morphisms $(x\stackrel{b}{\ri}y), (x'\stackrel{b'}{\ri}y')$ in $\mathbb P$, one defines
\begin{equation}\label{mt}
(x\stackrel{b}{\ri}y)\otimes(x'\stackrel{b'}{\ri}y')=(x+x'\stackrel{b+b'}{\rightarrow}y+y').
\end{equation}
Associativity, commutativity and unit constraints  are identities (${\bf a}=id, {\bf c}=id, {\bf l}= id = {\bf r}$).
By the definition of an abelian crossed module, it is easy to check that
 $\mathbb P$ is a strict Picard category.

$\bullet $ Conversely, for a strict Picard category $(\mathbb P,\otimes)$, we determine an {\it associated} abelian crossed module $\mathcal M_{\mathbb P}=(B, D,d)$ as follows.
Set
\begin{equation} D= \mathrm{Ob}(\mathbb P),\;\;B=\{x\xrightarrow{b}0 | x\in D  \}. \notag \end{equation}
The  operations in $D$ and  $B$  are, respectively,  given by
$$x+y=x\otimes y,\;\ b+c=b\otimes c. $$
Then $D$ becomes an abelian group whose zero element is 0, and the
inverse of $x$ is $-x$ ($x\otimes(-x)=0$).
$B$ is a group whose zero element is $id_0$, and the
inverse of $(x\xrightarrow{b}  0)$
 is the morphism $(-x\xrightarrow{\overline{b}}  0 )(b\otimes \overline{b}=id_0)$. Further, $B$ is abelian due to the naturality of the commutativity constraint  {\bf c}= $id$.

The homomorphism $d:B\rightarrow D$ is given by
$$d(x\xrightarrow{b}  0)=x.$$

\noindent{\bf Definition.} A {\it homomorphism}
$(f_1,f_0):(B, D,d)\ri (B',D',d')$ of abelian crossed modules consists of group homomorphisms  $f_1:B\ri B'$, $f_0:D\ri D'$ such that
\begin{equation*}\label{gr1}
f_0d=d'f_1.
\end{equation*}

Clearly, the category of abelian crossed modules is a full subcategory of the category of crossed modules.

In order to classify abelian crossed modules we establish the following lemmas.

\begin{lem}\label{t1}
Let $(f_1,f_0): \mathcal M = (B,D,d)\ri \mathcal M' = (B',D',d')$ be a homomorphism
of abelian crossed modules. Let $\mathbb P$ and $\mathbb P'$ be  strict
Picard categories  associated to  $\mathcal M$ and $\mathcal M'$, respectively.\\
\indent $i)$ There exists the functor $F:\mathbb P\ri\mathbb
P'$  defined by
$F(x)=f_0(x),\;F(b)=f_1(b),$ for $x\in D,  b\in B$.\\
\indent $ii)$ Natural isomorphisms $\widetilde{F}_{x,y}:F(x)+F(y)\ri F(x+y)$ together with $F$ is a
symmetric monoidal functor if and only if
$\widetilde{F}_{x,y}=\varphi(\overline{x},\overline{y}) $, where $\varphi $ is a symmetric 2-cocycle of the group $Z^2_s( \Coker d, \Ker d').$
\end{lem}
\begin{proof}
i) By the determination of a strict Picard category associated to an
abelian crossed module and the fact that  $f_1$ is a homomorphism,
$F$ is a functor.

ii) Since $f_1, f_0$ are  group homomorphisms, for two morphisms $(x\stackrel{b}{\ri}x')$,
$(y\stackrel{c}{\ri}y')$ in $\mathbb P$, we have
\begin{equation*}\label{tx}F(b\otimes c)=F(b)\otimes F(c).
\end{equation*}

On the other hand, since $f_0$ is a homomorphism and $F(x)=f_0(x)$, $\widetilde{F}_{x,y}: F(x)+F(y)\rightarrow F(x+y)$
is a morphism in $\mathbb P'$ if and only if
 $d'(\widetilde{F}_{x,y})=0'$, i.e.,
$$\widetilde{F}_{x,y}\in \Ker d'.$$
Then the naturality of  $(F, \widetilde{F})$, that is the commutativity of the following diagram
\begin{equation*}\label{bdtn}
\begin{diagram}
\node{F(x)+F(y)}\arrow{e,t}{\widetilde{F}_{x,y}}\arrow{s,l}{F(b)\otimes
F(c)}\node{F(x+y)}\arrow{s,r}{F(b\otimes c)}\\
\node{F(x')+F(y')}\arrow{e,b}{\widetilde{F}_{x',y'}}\node{F(x'+y'),}
\end{diagram}
\end{equation*}
is equivalent to the relation
 $\widetilde{F}_{x,y}=\widetilde{F}_{x',y'}$, where $x=d(b)+x', y=d(c)+y'.$
 This determines a function $\varphi: \Coker d\times \Coker d\rightarrow \Ker d'$  by
 $$ \varphi(\overline{x}, \overline{y})=\widetilde{F}_{x,y}.$$

 By $F(0)=0'$, the compatibility of $(F,\widetilde{F})$ with  unit constraints  is equivalent to the normalization of $\varphi$. From the relations \eqref{htmt} and \eqref{mt},
   the compatibility of
$(F,\widetilde{F})$ with  associativity, commutativity constraints are, respectively, equivalent to  relations
\begin{equation*} \widetilde{F}_{y,z}+\widetilde{F}_{x,y+z}=\widetilde{F}_{x,y}+\widetilde{F}_{x+y,z}, \end{equation*} \begin{equation*}  \widetilde{F}_{x,y}=\widetilde{F}_{y,x}. \end{equation*}
This shows that $\varphi\in Z^2_s( \Coker d, \Ker d')$.
\end{proof}

\noindent\textbf{Definition.} A symmetric monoidal functor
$(F,\widetilde{F}):\mathbb P\ri\mathbb P'$ between Picard
categories $\mathbb P, \mathbb P'$ is termed  {\it regular  } if
$F(x)\otimes F(y)=F(x\otimes y)$ for  $x,y \in $ Ob$\mathbb P$.

\vspace{0.1cm}
 Thanks to Lemma \ref{t1}, we  determine the category ${\bf AbCross}$ whose objects are abelian crossed modules and morphisms are triples $(f_1, f_0, \varphi)$, where $(f_1, f_0): (B\xrightarrow{d} D)\rightarrow
  (B'\xrightarrow{ d'}D')$ is a homomorphism of abelian crossed modules and $\varphi\in Z^2_s( \Coker d, \Ker d')$.

\begin{lem}\label{n1}
Let $\mathbb P$, $\mathbb P'$  be corresponding strict Picard
categories associated to  abelian crossed modules $(B, D,d)$, $(B',
D',d')$, and  let $(F,\widetilde{F}):\mathbb P \ri \mathbb P'$ be a
regular  symmetric monoidal functor. Then, the triple
$(f_1,f_0,\varphi)$, where
$$f_1(b)=F(b),\ f_0(x)=F(x),\ \varphi(s_1,s_2)=\widetilde{F}_{x_1,x_2},$$
for $b\in B,\ x\in D,\ x_i \in s_i \in \Coker d, i=1,2$, is a morphism in ${\bf AbCross}$.
\end{lem}
\begin{proof}
Since $F$ is regular, $f_0$ is a group homomorphism. Since $F$
preserves the composition of morphisms, $f_1$ is a group
homomorphism.

Any $b\in B$  can be considered as a morphism $(db\stackrel{b}{\ri} 0)$
in $\mathbb P$, and hence $(F(db)\stackrel{F(b)}{\ri}0')$ is a morphism
in $\mathbb P'$. This means that $f_0(d(b))=d'(f_1(b))$, for all $b\in B.$  Thus, $(f_1, f_0)$ is a homomorphism of abelian crossed modules.

 According to Lemma \ref{t1}, $\widetilde{F}_{x_1,x_2} $ determines a function $\varphi\in Z^2_s( \Coker d, \Ker d')$ by
      $$\varphi (s_1,s_2)=\widetilde{F}_{x_1,x_2}, \;\; x_i\in s_i\in  \Coker d,\ i=1,2.$$
  Therefore, $(f_1, f_0,\varphi)$ is a morphism in ${\bf AbCross}$.
 \end{proof}
Let ${\bf Picstr}$ denote the category of strict Picard categories
and regular  symmetric monoidal functors, we obtain the following
theorem.
\begin{thm}[\label{pl}Classification Theorem] There exists an equivalence
\[\begin{matrix}
 \Phi:{\bf AbCross}&\ri& \bf Picstr,\\
(B\ri D)&\mapsto&\mathbb{P}_{B\ri D},\\
(f_1,f_0,\varphi)&\mapsto&(F,\widetilde{F}),
\end{matrix}\]
where $F(x)=f_0(x),\ F(b)=f_1(b),\
\widetilde{F}_{x_1,x_2}=\varphi(s_1,s_2)$ for $x\in
D,\    b\in B,\ x_i \in s_i \in \Coker d,\ i=1,2.$
\end{thm}
\begin{proof}
 Suppose that $\mathbb P$ and $\mathbb P'$ are Picard categories associated to abelian crossed modules ${B\ri D}$ and $ {B'\ri D'}$, respectively.
By Lemma \ref{t1}, the correspondence  $(f_1, f_0, \varphi)\mapsto (F,\widetilde{F})$ determines an injection on
the homsets
$$\Phi:\Hom_{{\bf AbCross}}(B\ri D, B'\ri D')\ri \Hom_{{\bf Picstr}}(\mathbb P_{B\ri D},\mathbb P_{B'\ri D'}).$$
According to Lemma \ref{n1},  $\Phi$ is  surjective.

 If $\mathbb P$ is a strict Picard category, and $\mathcal M_{\mathbb P}$ is an abelian crossed module associated to it, then $\Phi(\mathcal M_{\mathbb P})=\mathbb P$ (not only isomorphic). Therefore, $\Phi$ is an equivalence.
\end{proof}
\noindent{\bf Remark.} Denoted by $\bf AbCross^\ast$ the subcategory of
 $\bf AbCross$ whose morphisms are homomorphisms of abelian crossed modules ($\varphi=0$), and denote by  $\bf Picstr^\ast$ the subcategory of $\bf Picstr$ whose morphisms are \emph{strict} symmetric monoidal functors ($\widetilde{F}=id$). Then these two categories are equivalent via $\Phi$. This result is analogous to  Theorem 1 \cite{Br76}.

\section {Classification of group extensions of the type of an abelian crossed module}

The concept of group extension of the type of a crossed module was introduced by Dedecker \cite{Ded} (see also \cite{Br94}). This concept has a version for abelian crossed modules as follows.

\vspace{0.1cm}
 \noindent\textbf{Definition.}
Let $\mathcal M= (B\xrightarrow{d} D)$ be an abelian crossed
module, and let $Q$ be an abelian group. An  \emph{abelian extension  of $B$
by $Q$ of  type  $\mathcal M$} is  the diagram of group homomorphisms
\begin{align}  \begin{diagram}\label{dng}
\xymatrix{\mathcal E:\;\;\;\;\;\; 0 \ar[r]& B \ar[r]^j \ar@{=}[d] &E  \ar[r]^p \ar[d]^\varepsilon & Q \ar[r]& 0, \\
& B \ar[r]^d & D}
\end{diagram}
\end{align}
where the top row is exact and $(id_B,\varepsilon)$ is a homomorphism of abelian crossed modules.

So, any  extension of the type of an abelian crossed module is an
  extension of the type of a crossed module. \vspace{5pt}

Two extensions $  \mathcal E, \mathcal E'$ of $B$ by $Q$ of  type  $\mathcal M$ are said to be \emph{equivalent }  if the following diagram commutes
\begin{align}
 \begin{diagram}\label{td}
\xymatrix{\mathcal E:\;\;0 \ar[r]& B \ar[r]^j \ar@{=}[d] &E  \ar[r]^p
\ar[d]^\alpha & Q \ar[r]\ar@{=}[d] & 0,&\;\;\;\;E
\ar[r]^\varepsilon&D \\
 \mathcal E':\;\;0 \ar[r]& B \ar[r]^{j'}  &E'  \ar[r]^{p'}   & Q \ar[r]&
 0,&\;\;\;\;E'
\ar[r]^{\varepsilon'}&D}
\end{diagram}
\end{align}
and $\varepsilon'\alpha=\varepsilon$. Obviously, $\alpha $ is
an isomorphism.

In the diagram
\begin{equation*} \label{bd10} \begin{diagram}
\xymatrix{\mathcal E:\;\; 0 \ar[r]& B \ar[r]^j \ar@{=}[d] &E  \ar[r]^p \ar[d]^\varepsilon & Q \ar[r]\ar@{.>}[d]^\psi & 0, \\
 & B \ar[r]^d  &D  \ar[r]^q   & \text{Coker}d }
\end{diagram}
\end{equation*}
since the top row is exact and $ q\circ \varepsilon \circ j = q\circ
d =0,$ there  is a homomorphism  $\psi: Q\rightarrow \Coker d $ such
that the right hand side square commutes. Moreover, $\psi$  is
dependent only on the equivalence class of the extension $\mathcal
E$.

Our objective is to study the set
$$\Ext^{ab}_{B\ri D}(Q,B,\psi)$$
 of equivalence
classes of extensions of $B$ by $Q$ of type
 $B\xrightarrow{d} D$ inducing $\psi:Q\rightarrow \Coker d$. It is well-known that
the set
$\Ext_{B\ri D}(Q,B,\psi)$ for extensions  of the type of a (not necessarily abelian) crossed module was classified by Brown and Mucuk.
In the present paper, we use the obstruction theory of symmetric monoidal functors
to prove Theorem \ref{dlc} which is an abelian analogue of Theorem 5.2 in \cite{Br94}.
  The second assertion of
this theorem can be  seen as a consequence of Schreier Theory
(Theorem \ref{schr}) due to symmetric monoidal functors between strict Picard categories
$\mathbb P_{B\ri D}$ and $\Dis Q$, where $\Dis Q$ is a Picard category of
type  $(Q,0,0)$.
\begin{lem}\label{bd7} Let   $B\xrightarrow{d} D$ be an abelian crossed module, $Q$ be an abelian group and
$\psi:Q\ri\Coker d$ be a group homomorphism. Then,
for each symmetric monoidal functor $(F,\widetilde{F}):\Dis Q\ri \mathbb P$ which satisfies $F(0)=0$ and induces the pair  $(\psi,0): (Q,0)\rightarrow (\Coker d, \Ker d)$, there exists an extension
$\mathcal E_F$ of  $B$ by $Q$ of type $B\ri
D$ inducing $\psi$.
\end{lem}
Such an extension $\mathcal E_F$ is called  {\it associated} to a symmetric monoidal functor $(F,\widetilde{F})$.
\begin{proof}
Suppose that $(F,\widetilde{F}):\; \Dis Q\ri \mathbb P$ is a symmetric monoidal functor.  Then, we set a function
  $f: Q\times Q \rightarrow
B$  as follows $$f(u, v) = \widetilde{F}_{u, v}.$$

Because $\widetilde{F}_{u,v}$ is  a morphism in $\mathbb P$, one has
\begin{equation*} F(u)+F(v)=df(u, v)+F(u+ v).\label{mt2}
\end{equation*}

Since $F(0)=0$ and   $(F,\widetilde{F})$ is compatible with the
strict constraints of Dis$Q$ and $\mathbb P$, $f$ is a normalized
function satisfying
\begin{align}
f(v,t)+ f(u,v+t)&=f(u,v)+ f(u+v,t), \label{eq1}\\
f(u,v)&= f(v,u). \label{eq1'}
\end{align}
Now we construct the semidirect product $E_0=[B,f,Q]$, that is, $E_0=B\times Q$
with the operation
$$(b,u)+(c,v)= (b+c + f(u,v),u+v).$$
The set $E_0$ is an abelian group due to the normalization of $f$
and the relations \eqref{eq1}, \eqref{eq1'},  the zero element is
$(0,0)$ and $-(b,u)=(-b -f(u,-u),-u)$. Then, we have an exact
sequence of abelian groups
 $$ \mathcal E_F:\;\;0\ri B\stackrel{j_0}{\rightarrow}E_0\stackrel{p_0}{\rightarrow}Q\ri 0,$$
 where
$$ j_0(b)=(b,0),\  p_0(b,u)= u, \ b\in B, u\in Q.$$
The map $\varepsilon:E_0\ri D$  given by
$$\varepsilon(b,u)=db+F(u),\;(b,u)\in E_0,$$
is a homomorphism, and hence the pair $(id_B,\varepsilon)$ is a
homomorphism of abelian crossed modules. 
Therefore,
one obtains an extension of the type of an abelian crossed module $\mathcal E_F$ satisfying  the diagram \eqref{dng}.
For all $u\in Q$, one has
 $$q\varepsilon(b,u)= q(db+F(u))= qF(u)=\psi(u),$$
i.e., this extension induces $\psi: Q\rightarrow \Coker d$.
\end{proof}
Under the assumptions of Lemma \ref{bd7}, we have

\begin{thm}[Schreier  Theory  for group extensions of the type of an abelian crossed module]\label{schr}
There exists a bijection
$$\Omega:\mathrm{Hom}^{Pic}_{(\psi,0)}[\mathrm{Dis}Q,\mathbb P_{B\rightarrow D}]\rightarrow \mathrm{Ext}^{ab}_{B\ri D}(Q, B,\psi)$$
if one of the above sets is nonempty.
\end{thm}
\begin{proof}
{\it Step 1: Symmetric monoidal functors  $(F,\widetilde{F}) $, $(F',\widetilde{F}') $ are homotopic if and only if the
corresponding associated extensions $\mathcal E_F, \mathcal E_{F}'$  are equivalent.}

  First, since every symmetric monoidal functor $(F,\widetilde{F}) $ is  homotopic to one $(G,\widetilde{G}) $ in which $G(0)=0$, the following symmetric monoidal functors are regarded as the functors which have this property .

Suppose that $F, F':$ Dis$Q\ri\mathbb P_{B\rightarrow D}$ are homotopic
by a homotopy $\alpha:F\ri F'$.  By Lemma \ref{bd7}, there exist the
extensions $\mathcal E_F$ and $\mathcal E_{F'}$ associated to $F$ and $F'$, respectively. Then, it follows from the definition of a homotopy that $\alpha_0=0$ and  the
following diagram commutes
\[\begin{diagram}\xymatrix{Fu+Fv\ar[r]^ {\widetilde{F}_{u,v}}\ar[d]_{\alpha_u\otimes
\alpha_v}&F(u+v)\ar[d]^{\alpha_{u+v}}\\
F'u+F'v\ar[r]_{\widetilde{F}'_{u,v}}&F'(u+v),
}\end{diagram}\]
that is,  $$\widetilde{F}_{u,v}+\alpha_{u+v}=\alpha_u\otimes
\alpha_v+\widetilde{F}'_{u,v}.$$
By the relation (\ref{mt}), one has
\begin{equation}\label{eq5}
f(u,v)+\alpha_{u+v}=\alpha_u+\alpha_v+f'(u,v),
\end{equation}
where  $f(u,v)=\widetilde{F}_{u,v}, f'(u,v)=\widetilde{F}'_{u,v}$. Now we set
\begin{align*} \alpha^\ast: E_F&\rightarrow E_{F'}\\
(b,u)&\mapsto (b+\alpha_u,u). \end{align*}
Then $\alpha^\ast$ is a homomorphism thanks to the relation (\ref{eq5}). Further, the diagram \eqref{td} commutes.
It remains to show that $\varepsilon'\alpha^\ast=\varepsilon$.
Since $\alpha:F\ri F'$ is a homotopy,
$F(u)=d(\alpha_u)+F'(u)$. Then,
\[\begin{aligned}\varepsilon'\alpha^*(b,u)&=\varepsilon'(b+\alpha_u,u)=d(b+\alpha_u)+F'(u)\\
\;\;&
=d(b)+d(\alpha_u)+F'(u)=d(b)+F(u)=\varepsilon(b,u).\end{aligned}\] Therefore,
 $\mathcal E_F$ and $ \mathcal E_{F'}$ are equivalent.

Conversely, if an isomorphism $\alpha^\ast:E_F\rightarrow E_{F'}$  satisfying the triple $(id_B, \alpha^*,id_Q)$ is an equivalence of extensions, then it is easy to see that
 $$\alpha^\ast(b,u)=(b+\alpha_u,u),$$ where  $\alpha:Q\ri
B$ is a function satisfying $\alpha_0=0$. By retracing our steps, $\alpha$ is a homotopy between $F$ and $F'$.


{\it Step 2: $\Omega$ is surjective.}

Assume that $\mathcal E$ is an extension $E$ of $B$ by $Q$
of type $B\rightarrow D$ inducing $\psi:
Q\rightarrow \Coker d$. We prove that $\mathcal E$ is equivalent to the semidirect product extension $\mathcal E_F$ which is associated to a symmetric monoidal functor
$(F,\widetilde{F}):\;\mathrm{Dis} Q\ri \mathbb P_{B\ri D}$ .

For any $u\in Q$, choose a representative $e_u\in E$ such that
$p(e_u)=u,\;e_0=0$. Each element of $E$ can be represented uniquely  as  $b+e_u $ for $b\in B, u\in Q$. The representatives  $\left\{e_u\right\}$ induces a normalized function $f:Q\times Q\ri B$
 by
\begin{equation}    e_u+e_v=f(u, v)+e_{u+v}.\label{eq3'}\end{equation}
Then,  the group structure of $E$ can be described by
\begin{equation*} (b+e_u)+(c+e_v)= b+c+f(u, v)+e_{u+v}.
   \end{equation*}

Now, we construct a symmetric monoidal functor   $(F,\widetilde{F}):\;
\Dis Q\ri \mathbb P$ as follows. Since $\psi(u)=\psi p(e_u)=q\varepsilon(e_u)$, $\varepsilon(e_u)$ is a representative of $\psi(u)$ in $D$. Thus, we set
 $$F(u)=\varepsilon(e_u),\   \widetilde{F}_{u,v}=f(u,v).$$
The relation  \eqref{eq3'} shows that $ \widetilde{F}_{u,v}$ are
actually  morphisms in $\mathbb P$.  Obviously, $F(0)=0.$ This
together with  the normalization condition of the function  $f$
implies the compatibility of $(F, \widetilde{F})$ with the unit
constraints. The associativity and commutativity  laws of the
operation in
 $E$ lead to the relations
   (\ref{eq1}), (\ref{eq1'}), respectively. These relations prove that $(F, \widetilde{F})$  is compatible with the associativity  and commutativity constraints of Dis$ Q$ and $\mathbb P$, respectively. The naturality of  $\widetilde{F}_{u,v}$ and the condition of $F$ preserving the composition of morphisms are obvious.

Finally, it is easy to check that  the semidirect product extension $\mathcal
E_F$ associated to $(F,\widetilde{F})$ is equivalent to the extension
$\mathcal E$ by the isomorphism  $ \beta : (b,u)\mapsto b+e_u$.
\end{proof}

\vspace{0.2cm}
Let $\mathbb P= \mathbb P_{B\rightarrow D}$
be a strict  Picard category associated to an abelian crossed module $B\rightarrow D$. Since $\pi_0(\mathbb P)=\Coker d$ and $\pi_1(\mathbb P)=\Ker d$, the reduced  Picard category $ S_{\mathbb P}$ is of form
$$ S_{\mathbb P}= (\mathrm{Coker} d, \mathrm{Ker} d,k), \; \overline{k}\in H^3_s (\mathrm{Coker} d, \mathrm{Ker} d).$$
Then, by the relation (\ref{ct}),  the pair of homomorphisms
$(\psi,0):(Q,0)\rightarrow (\mathrm{Coker} d, \Ker d)$ induces an
{\it obstruction}
 $$\psi^\ast k\in Z^3_s(Q, \Ker d).$$
Under this notion of obstruction, we state and prove the following
theorem.
\begin{thm}\label{dlc} Let  $(B, D,d)$ be an abelian crossed module, and let
$\psi: Q\rightarrow \Coker d$ be a  homomorphism of abelian groups. Then, the vanishing of $\overline{\psi^\ast k}$ in $H^3_s (Q, \Ker d)$ is  necessary and sufficient for there to exist an extension of $B$ by $Q$ of type $B\rightarrow D$ inducing $\psi$. Further, if $\overline{\psi^\ast k}$ vanishes, then the set of equivalence classes of such extensions is bijective with
$H^2_s(Q,\Ker d)$.
\end{thm}
\begin{proof}
By the assumption $\overline{\psi^\ast k}=0$, it follows from by Proposition \ref{md1} that
there is a symmetric monoidal functor $(\Psi,\widetilde{\Psi}):\Dis Q\ri S_{\mathbb P}$.
 Then the composition of $(\Psi,\widetilde{\Psi})$ and the canonical symmetric monoidal functor
$(H,\widetilde{H}): S_{\mathbb P}\ri \mathbb P$ is a  symmetric monoidal functor
$(F,\widetilde{F}):\Dis Q\ri \mathbb P$, and hence by Lemma \ref{bd7}, we obtain an associated extension
$\mathcal E_F$.

Conversely, suppose that there is an extension as in the diagram \eqref{bd10}. Let $ \mathbb P'$  be a strict Picard category associated to the
abelian crossed module $B\ri E$. Then, according to Lemma \ref{t1}, there is a
  symmetric monoidal functor   $F:\mathbb P' \ri \mathbb P$. Since the reduced Picard category
of $\mathbb P'$  is $\Dis Q$,  by Proposition  \ref{md1} $i)$, $F$ induces a
symmetric monoidal functor of type  $(\psi,0)$ from $\Dis Q$ to $S_\mathbb P=(\Coker d,\Ker d,k)$. Now,
 thanks to Proposition \ref{md1} $iii)$, the obstruction of the pair $(\psi,0)$  vanishes in $H^3_s(Q, \Ker d),$ i.e.,
  $\overline{\psi^\ast k}=0$.

The final assertion  of the theorem  is obtained  from Theorem
\ref{schr}. First, there is a natural bijection
$$\Hom^{Pic}_{(\psi,0)}[\Dis Q, \mathbb P]\leftrightarrow \Hom^{Pic}_{(\psi,0)}[\Dis Q, S_{\mathbb P}].$$
Since $\pi_0(\Dis Q)=Q, \pi_1(S_{\mathbb P})=\Ker d$, the bijection
$$\mathrm{Ext}^{ab}_{B\ri D}(Q, B,\psi)\leftrightarrow
H^2_s(Q,\Ker d)$$
follows from Theorem \ref{schr} and Proposition \ref{md1}.
\end{proof}

In the case when the homomorphism  $d$ of the abelian crossed module $\mathcal M$ is a monomorphism, then the diagram
 \eqref{bd10} shows that the extension ($\mathcal E:B\ri E\ri Q$) is obtained from the extension
($\mathcal D:B\ri D\ri \Coker d$) and $\psi$, i.e.,
$\mathcal E=\mathcal D \psi$ (see \cite{MacL63, Fu}).
  Since $\Ker d=0$, by Theorem \ref{dlc}, we obtain a well-known result as follows.
\begin{hq} Let ($\mathcal D: B\ri D\ri C$) be an  extension of abelian groups and  $\psi: Q\ri C$ be a
 homomorphism of abelian groups. Then, there is an extension $\mathcal D \psi$  determined uniquely up to equivalence.
\end{hq}

{\bf Acknowledgement} The authors are much indebted to the referee, whose useful observations greatly improved our exposition.

\begin{center}
{}
\end{center}

\end{document}